\documentclass[12pt]{article}
\usepackage{amssymb}
\usepackage{mathrsfs}%有更多的数学符号，比如花体$\mathscr P$
\usepackage{tikz}
\usepackage{mathrsfs}
\usepackage{amssymb}
\usepackage{amsmath}
\usepackage{bm}
\usepackage{graphicx}
\usepackage{pst-poly}  %画图
\usepackage{pst-plot}  %坐标
\usepackage{url}  %\url{输入你的网址，直接用~即可}

\newtheorem{theorem}{Theorem}[section]
\newtheorem{lemma}[theorem]{Lemma}
\newtheorem{definition}[theorem]{Definition}
\newtheorem{corollary}[theorem]{Corollary}
\newtheorem{proposition}[theorem]{Proposition}

\newtheorem{example}[theorem]{Example}
\newtheorem{remark}[theorem]{Remark}

\newtheorem{problem}[theorem]{Problem}

\newtheorem{asp}[theorem]{Assumption}
\newenvironment{proof}[1][Proof]{\noindent\textbf{#1.} }{\hfill\rule{1mm}{2mm}}

\makeatletter \@addtoreset{equation}{section} \makeatother

\begin{document}

\title
{On restricted edge-connectivity of replacement product
graphs\footnote{This paper has been accepted for publication in
SCIENCE CHINA Mathematics. }}

\author
{HONG ZhenMu \\
{\small School of Finance}\\
{\small Anhui University of Finance \& Economics}\\
{\small Bengbu, Anhui, 233030, China}\\
{\small Email: zmhong@mail.ustc.edu.cn}\\
\\
XU JunMing\thanks{Corresponding author}\\
{\small School of Mathematical Sciences}\\
{\small University of Science and Technology of China}\\
{\small Hefei, Anhui, 230026, China}\\
{\small Email: xujm@ustc.edu.cn}\\}

\date{}
\maketitle

\begin{center}
\begin{minipage}{140mm}

\begin{center} {\bf Abstract} \end{center}

This paper considers the edge-connectivity and the restricted
edge-connectivity of replacement product graphs, gives some bounds
on edge-connectivity and restricted edge-connectivity of replacement
product graphs and determines the exact values for some special
graphs. In particular, the authors further confirm that under
certain conditions, the replacement product of two Cayley graphs is
also a Cayley graph, and give a necessary and sufficient condition
for such Cayley graphs to have maximum restricted edge-connectivity.
Based on these results, the authors construct a Cayley graph with
degree $d$ whose restricted edge-connectivity is equal to $d+s$ for
given odd integer $d$ and integer $s$ with $d \geqslant 5$ and
$1\leqslant s\leqslant d-3$, which answers a problem proposed ten
years ago.

\vskip16pt



\end{minipage}
\end{center}

\newpage

\section{Introduction}

Throughout this paper, we follow~\cite{x03} for graph-theoretical
terminology and notation not defined here. Specially, $G=(V, E)$ is
a simple connected undirected graph, where $V=V(G)$ is the
vertex-set of $G$ and $E=E(G)$ is the edge-set of $G$; $d_G(x)$ is
the degree of a vertex $x$ in $G$, the number of edges incident with
$x$ in $G$; $\delta(G)=\min\{d_G(x): x\in V(G)\}$ is the minimum
degree of $G$; $\xi(G)=\min\{d_G(x)+d_G(y)-2: xy\in E(G)\}$ is the
minimum edge-degree of $G$.

The connectivity $\kappa(G)$ (resp. edge-connectivity $\lambda(G)$)
of $G$ is defined as the minimum number of vertices (resp. edges)
whose removal results in disconnected. The well-known Whitney
inequality states that $\kappa (G)\leqslant \lambda (G)\leqslant
\delta (G)$ for any graph $G$. In this paper, we are interested in
the edge-connectivity $\lambda(G)$.

It is well known that when the underlying topology of an
interconnection network is modeled by a connected graph $G=(V,E)$,
where $V$ is the set of processors and $E$ is the set of
communication links in the network, the edge-connectivity
$\lambda(G)$ of $G$ is an important measurement for reliability and
fault tolerance of the network since the larger $\lambda(G)$ is, the
more reliable the network is. However, when computing $\lambda(G)$,
one implicitly assumes that all edges incident with the same vertex
may fail simultaneously. Consequently, this measurement is
inaccurate for large-scale processing systems in which some subsets
of system links can not fail at the same time in real applications.

To overcome the shortcomings of edge-connectivity, Esfahanian and
Hakimi~\cite{eh88} proposed the concept of the restricted
edge-connectivity $\lambda'(G)$ of a graph $G$, which is the minimum
number of edges whose removal results in disconnected and no
isolated vertices, and gave the following result.

\begin{theorem}[{See~\cite{eh88}}]\label{thm1.1}\quad
$\lambda(G)\leqslant \lambda'(G)\leqslant \xi(G)$ for any graph $G$
of order $n (\geqslant 4)$ except for a star $K_{1,n-1}$.
\end{theorem}

A graph $G$ is vertex-transitive if for any two vertices $x$ and $y$
in $G$, there is a $\sigma\in$ Aut\,$(G)$ such that $y=\sigma(x)$,
where Aut\,$(G)$ is the automorphism group of $G$. Clearly,
$\xi(G)=2d-2$ for a vertex-transitive connected graph $G$ with
degree $d$. Xu {\it et al.} obtained the following results.

\begin{theorem}[{See~\cite{xx02}}]\label{thm1.2}\quad
Let $G$ be a vertex-transitive connected graph with order $n\
(\geqslant 4)$ and degree $d\ (\geqslant 2)$. Then

$(a)$\ $\lambda'(G)=\xi(G)=2d-2$ if $n$ is odd or $G$ contains no
triangles;

$(b)$\ there exists an integer $m\ (\geqslant 2)$ such that
$d\leqslant \lambda'(G) =\frac nm\leqslant 2d-3$ otherwise.
\end{theorem}

\begin{theorem}[{See~\cite{lx04}}]\label{thm1.3}\quad
For any given integers $d$ and $s$ with $d\geqslant 3$ and
$0\leqslant s\leqslant d-3$, there is a connected vertex-transitive
graph $G$ with degree $d$ and $\lambda'(G)=d+s$ if and only if
either $d$ is odd or $s$ is even.
\end{theorem}

In~\cite{lx04}, for any odd integer $d (\geqslant 3)$ and any
integer $s$ with $0\leqslant s\leqslant d-3$, the authors construct
a vertex-transitive graph $G$ with degree $d$ and
$\lambda'(G)=d+s=\frac 12\,n$. Note that the condition ``$d\leqslant
\lambda'(G)\leqslant 2d-3$" implies $\lambda'(G)=d$ if $d=3$. By
Theorem~\ref{thm1.2}, if a vertex-transitive graph $G$ is not
$\lambda'$-optimal, then $d\leqslant \lambda'(G) \leqslant
\frac{n}{2}$. Thus, a quite natural problem is proposed as follows
(see Conjecture 1 in Xu~\cite{x03b}).

\begin{problem}\label{prob1.4}\quad
Given an odd integer $d\,(\geqslant 5)$ and any integer $s$ with
$1\leqslant s\leqslant d-3$, whether or not there is a
vertex-transitive graph $G$ with order $n$ and degree $d$ such that
$\lambda'(G)=d+s<\frac 12\,n$.
\end{problem}

In this paper, we answer this question confirmedly by constructing a
Cayley graph, which is the replacement product of two Cayley graphs.

We will discuss the restricted edge-connectivity of a replacement
product graph in this paper. The rest of this paper is organized as
follows. In Section 2, we give some definitions with related
results. In Section 3, we establish the bounds on the
edge-connectivity for a replacement product graph and determine
exact values under some special conditions. In Section 4, we give
the lower and upper bounds on restricted edge-connectivity for
replacement product graphs and determine exact values under some
given conditions. In Section 5, we focus on Cayley graphs and
further confirm that under certain conditions, the replacement
product of two Cayley graphs is still a Cayley graph, and give a
necessary and sufficient condition for such Cayley graphs to have
maximum restricted edge-connectivity. Based on these results, we
construct a Cayley graph to answer Problem~\ref{prob1.4}
confirmedly. A conclusion is in Section 6.

\section{Preliminaries}

We first introduce the concept of the restricted edge-connectivity,
proposed by Esfahanian and Hakimi~\cite{eh88}, stated here slightly
different from theirs.

Let $G$ be a non-trivial connected graph and $F\subset E(G)$. If
$G-F$ is disconnected and contains no isolated vertices, then $S$ is
called a restricted edge-cut of $G$. The restricted
edge-connectivity of $G$, denoted by $\lambda'(G)$, is defined as
the minimum cardinality over all restricted edge-cuts of $G$.
Esfahanian and Hakimi~\cite{eh88} proved $\lambda'(G)$ is
well-defined for any connected graph $G$ of order $n (\geqslant 4)$
except for a star $K_{1,n-1}$. A graph $G$ is {\it
$\lambda'$-connected} if $\lambda'(G)$ exists, and a restricted
edge-cut $F$ is a {\it $\lambda'$-cut} if $|F|=\lambda'(G)$. A
$\lambda'$-connected graph is {\it $\lambda'$-optimal} if
$\lambda'(G) = \xi(G)$, and {\it not $\lambda'$-optimal} otherwise.
It is clear that if $G$ is a $\delta$-regular and $\lambda'$-optimal
graph of order $n$, then $\lambda(G)=\delta(G)=\delta$ and
$n\geqslant 4$.

The restricted edge-connectivity provides a more accurate measure of
fault-tolerance of networks than the edge-connectivity
(see~\cite{e89, eh88}). Thus, determining the value of $\lambda'$
for some special classes of graphs or characterizing
$\lambda'$-optimal graphs have received considerable attention in
the literature (see, for instance,
\cite{eh88,hv04,hv05,ll99,m03,mj02,wl02,wl99,wl01}).

%{\red For a graph $G$, denoted by Aut\,$(G)$ the automorphism group
%of $G$, it is vertex-transitive} if for any two vertices $x$ and $y$
%in $G$, there is a $\sigma\in$ Aut\,$(G)$ such that $y=\sigma(x)$.

Let $\Gamma$ be a finite group, and let $S$ be a subset of $\Gamma$
not containing the identity element of $\Gamma$. The {\it Cayley
graph} $C_\Gamma(S)$ is the graph having vertex-set $\Gamma$ and
edge-set $\{xy:\ x^{-1}y\in S, x,y\in \Gamma\}$.

Generally speaking, $C_\Gamma(S)$ is a digraph. The following result
is well-known (see, for instance, Xu~\cite{x01}).

\begin{lemma}\label{lem2.1}\quad
Cayley graphs are vertex-transitive and the Cartesian product of
Cayley graphs is a Cayley graph.
\end{lemma}

If $S=S^{-1}$, then $C_\Gamma(S)$ is an undirected graph. We are
interested in undirected graphs in this paper.

We now introduce two classes of Cayley graphs, because of their
excellent features, they are the most popular, versatile and
efficient topological structures of interconnection networks (see,
for instance, Xu~\cite{x01}).

\begin{example}\label{exa2.2}\quad {\rm
A {\it circulant graph $G(n;\pm S)$}, where
$S=\{s_1,s_2,\dots,s_k\}\subseteq \{1,2,\ldots,$ $\lfloor \frac
12n\rfloor\}$ with $s_1<s_2<\ldots<s_k$ and $n\geqslant 3$, has
vertex-set $V=\{0,1,\ldots, n-1\}$ and edge-set $E=\{ij$:
$|j-i|\equiv$ $s_i\,({\rm mod}\ n)\ \text{for some}\ s_i\in S\}$.

Clearly, $G(n;\pm 1)$ is a cycle $C_n$ and $G(n;\pm
\{1,2,\ldots,\lfloor \frac 12n\rfloor\})$ is a complete graph $K_n$.
The two graphs shown in Figure~\ref{f1} are $G(8;\pm \{1,3\})$ and
$G(8;\pm \{1,3,4\})$.

Note that the identity element of the ring group $\mathbb{Z}_n
(n\geqslant 2)$ is just the zero element, and the inverse of any
$i\in \mathbb{Z}_n$ is $n-i$. If let $S\subseteq \{1,2,\ldots,n-1\}$
and $S^{-1}=S$, then Cayley graph $C_{\mathbb{Z}_n}(S)$ is a
circulant graph $G(n;S)$ if $ n\geqslant 3$, and
$C_{\mathbb{Z}_2}(S)=K_2$. Thus, circulant graphs are
vertex-transitive by Lemma~\ref{lem2.1}.

Li and Li~\cite{ll98} showed that $G(n; \pm S)$ is
$\lambda'$-optimal and $\lambda'(G(n; \pm S))=4k-2$ if $k\geqslant
2$ and $s_k<\frac{n}{2}$.
 }
\end{example}

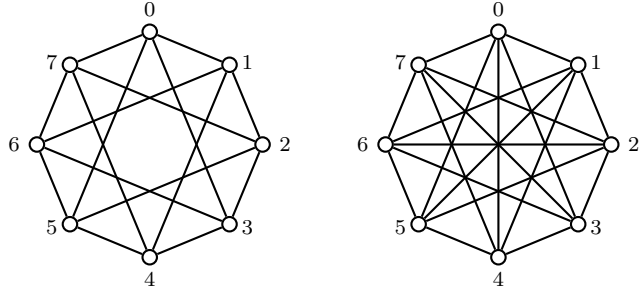
\begin{figure}[h]
\begin{pspicture}(-5,-1.7)(-1.5,1.8)
\rput{90}{%
\SpecialCoor\degrees[8]
  \multido{\i=0+1}{8}{\cnode(1.5;\i){.11}{1\i}}
 \ncline{10}{17}\ncline{10}{11}\ncline{11}{12}\ncline{12}{13}\ncline{13}{14}\ncline{14}{15}\ncline{15}{16}\ncline{16}{17}
 \ncline{10}{13}\ncline{10}{15}\ncline{11}{14}\ncline{11}{16}\ncline{12}{15}\ncline{12}{17}
 \ncline{13}{16}\ncline{14}{17}
  }%
 \rput(0,1.8){\scriptsize$0$}\rput(0,-1.8){\scriptsize$4$}
 \rput(-1.3,1.1){\scriptsize$7$}\rput(1.3,1.1){\scriptsize$1$}
 \rput(-1.8,0){\scriptsize$6$}\rput(1.8,0){\scriptsize$2$}
 \rput(-1.3,-1.1){\scriptsize$5$}\rput(1.3,-1.1){\scriptsize$3$}
 \end{pspicture}
 \begin{pspicture}(-6,-1.7)(-1.5,1.8)
\rput{90}{%
\SpecialCoor\degrees[8]
  \multido{\i=0+1}{8}{\cnode(1.5;\i){.11}{1\i}}
 \ncline{10}{17}\ncline{10}{11}\ncline{11}{12}\ncline{12}{13}\ncline{13}{14}\ncline{14}{15}\ncline{15}{16}\ncline{16}{17}
 \ncline{10}{13}\ncline{10}{14}\ncline{10}{15}
 \ncline{11}{14}\ncline{11}{15}\ncline{11}{16}
 \ncline{12}{15}\ncline{12}{16}\ncline{12}{17}
 \ncline{13}{16}\ncline{13}{17}\ncline{14}{17}
  }%
 \rput(0,1.8){\scriptsize$0$}\rput(0,-1.8){\scriptsize$4$}
 \rput(-1.3,1.1){\scriptsize$7$}\rput(1.3,1.1){\scriptsize$1$}
 \rput(-1.8,0){\scriptsize$6$}\rput(1.8,0){\scriptsize$2$}
 \rput(-1.3,-1.1){\scriptsize$5$}\rput(1.3,-1.1){\scriptsize$3$}
 \end{pspicture}
\caption{
\label{f1}                                               %f1
\footnotesize  (a)\ $G(8;\pm \{1,3\})$;\ (b)\ $G(8;\pm \{1,3,4\})$}
\end{figure}

\begin{example}\label{exa2.3}\quad {\rm
The {\it hypercube} $Q_n$ has the vertex-set consisting of $2^n$
binary strings of length $n$, two vertices being linked by an edge
if and only if they differ in exactly one position. Hypercubes
$Q_1,Q_2, Q_3$ and $Q_4$ are shown in Figure~\ref{f2}.

\begin{figure}[h]
\psset{unit=0.9}
\begin{pspicture}(-3.3,-.3)(0,4)
\cnode(1,1){.1}{0}\rput(.75,1){\scriptsize0}
\cnode(1,3){.1}{1}\rput(.75,3){\scriptsize1}
\ncline{0}{1}\rput(1,.5){\scriptsize$Q_1$}
\end{pspicture}
\begin{pspicture}(-1.5,-.3)(2,3)
\cnode(1,1){.1}{00}\rput(.7,1){\scriptsize00}
\cnode(1,3){.1}{10}\rput(.7,3){\scriptsize10}
\cnode(3,1){.1}{01}\rput(3.35,1){\scriptsize01}
\cnode(3,3){.1}{11}\rput(3.3,3){\scriptsize11}
\ncline{00}{01}\ncline{01}{11}\ncline{11}{10}\ncline{10}{00}
\rput(2,.5){\scriptsize$Q_2$}
\end{pspicture}
\begin{pspicture}(-1.5,-.3)(2,3)
\cnode(1,1){.1}{000}\rput(.6,1){\scriptsize000}
\cnode(1,3){.1}{001}\rput(.6,3){\scriptsize001}
\cnode(3,1){.1}{100}\rput(3.4,1){\scriptsize100}
\cnode(3,3){.1}{101}\rput(3.4,3){\scriptsize101}
\cnode(1.7,1.7){.1}{010}\rput(2.1,2.){\scriptsize010}
\cnode(1.7,3.7){.1}{011}\rput(1.3,3.8){\scriptsize011}
\cnode(3.7,1.7){.1}{110}\rput(4.1,1.8){\scriptsize110}
\cnode(3.7,3.7){.1}{111}\rput(4.1,3.8){\scriptsize111}
\ncline{000}{001}\ncline{001}{101}\ncline{101}{100}\ncline{100}{000}
\ncline{010}{011}\ncline{011}{111}\ncline{111}{110}\ncline{110}{010}
\ncline{000}{010}\ncline{001}{011}\ncline{101}{111}\ncline{100}{110}
\rput(2,.5){\scriptsize$Q_3$}
\end{pspicture}
\vskip2pt
\begin{pspicture}(-4.2,0.5)(5,4)
\cnode(1,1){.1}{0000}\rput(.55,1){\scriptsize0000}
\cnode(1,3){.1}{0100}\rput(.55,3){\scriptsize0100}
\cnode(3,1){.1}{0001}\rput(2.62,1.2){\scriptsize0001}
\cnode(3,3){.1}{0101}\rput(2.62,3.2){\scriptsize0101}
\cnode(1.9,1.8){.1}{0010}\rput(1.45,1.85){\scriptsize0010}
\cnode(1.9,3.8){.1}{0110}\rput(1.45,3.85){\scriptsize0110}
\cnode(3.9,1.8){.1}{0011}\rput(3.51,2){\scriptsize0011}
\cnode(3.9,3.8){.1}{0111}\rput(3.51,4){\scriptsize0111}
\cnode(5,1){.1}{1001}\rput(5.4,.8){\scriptsize1001}
\cnode(5,3){.1}{1101}\rput(5.4,2.8){\scriptsize1101}
\cnode(7,1){.1}{1000}\rput(7.4,.85){\scriptsize1000}
\cnode(7,3){.1}{1100}\rput(7.4,2.85){\scriptsize1100}
\cnode(5.9,1.8){.1}{1011}\rput(6.3,2){\scriptsize1011}
\cnode(5.9,3.8){.1}{1111}\rput(6.3,4){\scriptsize1111}
\cnode(7.9,1.8){.1}{1010}\rput(8.34,1.8){\scriptsize1010}
\cnode(7.9,3.8){.1}{1110}\rput(8.34,3.8){\scriptsize1110}
\ncline{0000}{0001}\ncline{0001}{0101}\ncline{0101}{0100}\ncline{0100}{0000}
\ncline{0010}{0011}\ncline{0011}{0111}\ncline{0111}{0110}\ncline{0110}{0010}
\ncline{0000}{0010}\ncline{0001}{0011}\ncline{0101}{0111}\ncline{0100}{0110}
\ncline{1001}{1000}\ncline{1000}{1010}\ncline{1010}{1011}\ncline{1011}{1001}
\ncline{1101}{1100}\ncline{1100}{1110}\ncline{1110}{1111}\ncline{1111}{1101}
\ncline{1001}{1101}\ncline{1000}{1100}\ncline{1010}{1110}\ncline{1011}{1111}
\nccurve[angleA=-20,angleB=-160]{0000}{1000}
\nccurve[angleA=-20,angleB=-160]{0010}{1010}
\nccurve[angleA=-20,angleB=-160]{0001}{1001}
\nccurve[angleA=-20,angleB=-160]{0011}{1011}
\nccurve[angleA=20,angleB=160]{0100}{1100}
\nccurve[angleA=20,angleB=160]{0110}{1110}
\nccurve[angleA=20,angleB=160]{0101}{1101}
\nccurve[angleA=20,angleB=160]{0111}{1111}
\rput(1.5,.4){\scriptsize$Q_4$}
\end{pspicture}
\caption{
\label{f2}                                       %f2
\footnotesize  The $n$-cubes $Q_1$, $Q_2$, $Q_3$ and $Q_4$}
\end{figure}
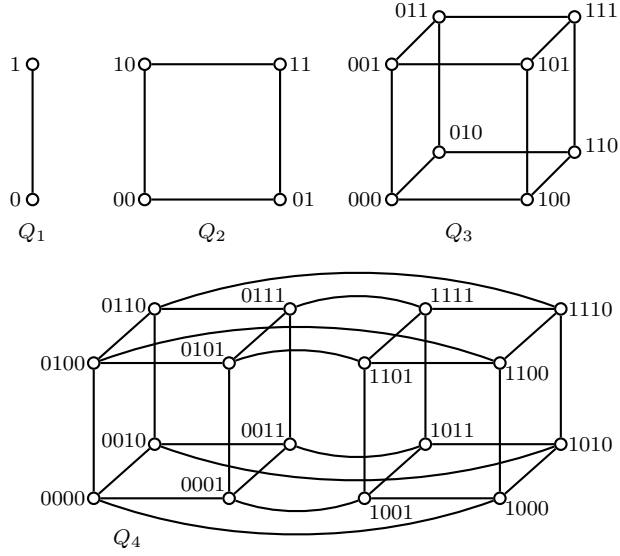

It is easy to see that the hypercube $Q_n$ is Cartesian products
$K_2\times K_2 \times \cdots \times K_2 $ of $n$ complete graph
$K_2$. Let $(\mathbb{Z}_2)^n=\mathbb{Z}_2\times \mathbb{Z}_2\times
\cdots \times \mathbb{Z}_2$ and
 \begin{equation}\label{e2.1}
e_0=\underbrace{0\cdots0}_{n}\ \ {\rm and}\ \
e_i=\underbrace{0\cdots0}_{i-1}1\underbrace{0\cdots0}_{n-i}\ \
\text{for each $i=1,2,\ldots,n$}.
 \end{equation}
Then $e_0$ is the identity element of $(\mathbb{Z}_2)^n$ and, by
Lemma~\ref{lem2.1}, $Q_n$ is a Cayley graph
$C_{(\mathbb{Z}_2)^n}(S)$ and so is vertex-transitive, where
$S=\{e_1, e_2,\cdots,e_n\}$, each of which is self-inverse and,
hence, $S=S^{-1}$.

Esfahanian~\cite{e89} showed that the hypercube $Q_n$ is
$\lambda'$-optimal, that is, $\lambda'(Q_n)=2n-2$ for $n\geqslant
2$. }
 \end{example}

Now, we introduce {\it the replacement product}. There are several
equivalent definitions of the replacement product proposed by
different authors (see~\cite{hlw06, rvw02}). Here, we adopt the
definition proposed by Hoory {\it et al.}~\cite{hlw06}. Let $G_1$ be
a $\delta_1$-regular graph on $n$ vertices and $G_2$ be a
$\delta_2$-regular graph on $\delta_1$ vertices. For every vertex
$x\in V(G_1)$, we label on all edges incident with $x$,
% in a clockwise order around $x$,
say $e^1_x,e^2_x,\dots, e^{\delta_1}_x$.

\begin{definition}\label{Def2.4}\quad
Let $G_1$ be a $\delta_1$-regular graph on $n$ vertices and $G_2$ be
a $\delta_2$-regular graph on $\delta_1$ vertices. The {\it
replacement product} of $G_1$ and $G_2$ is a graph, denoted by
$G_1\textregistered G_2$, where $V(G_1\textregistered
G_2)=V(G_1)\times V(G_2)$, two distinct vertices $(x,i)$ and
$(y,j)$, where $x,y\in V(G_1)$ and $i,j\in V(G_2)$, are linked by an
edge in $G_1\textregistered G_2$ if and only if either $x=y$ and
$ij\in E(G_2)$, or $xy\in E(G_1)$ and $e_x^i=xy=e_y^j$.
\end{definition}

Figure~\ref{f3} shows the replacement product of $K_4$ and $C_3$
with given labelling of edges around vertices of $K_4$.

\begin{figure}[h]
\begin{center}
%\psset{unit=.8}
\begin{pspicture}(-6,-2.5)(6,3.5)
\psset{radius=1.2}

%--------$K_4$------------
\cnode(-5,0){3pt}{a} \cnode(-3.268,-1){3pt}{b} \cnode(-5,2){3pt}{c}
\cnode(-6.732,-1){3pt}{d}
\ncline[linewidth=1.6pt]{a}{b}\ncline[linewidth=1.6pt]{a}{c}\ncline[linewidth=1.6pt]{a}{d}
\ncline[linewidth=1.6pt]{b}{c}\ncline[linewidth=1.6pt]{b}{d}\ncline[linewidth=1.6pt]{c}{d}

 \rput(-5,2.3){\scriptsize $x$}\rput(-7,-1.){\scriptsize $y$}\rput(-5,-.3){\scriptsize $z$}\rput(-3,-1.){\scriptsize $u$}
 \rput(-5.5,1.6){\scriptsize $e_x^0$}\rput(-4.8,1.2){\scriptsize $e_x^2$}\rput(-4.45,1.6){\scriptsize $e_x^1$}
 \rput(-6.6,-0.2){\scriptsize $e_y^0$}\rput(-6,-0.3){\scriptsize $e_y^1$}\rput(-6.2,-1.3){\scriptsize $e_y^2$}
 \rput(-4.75,.4){\scriptsize $e_z^0$}\rput(-5.5,-0.){\scriptsize $e_z^1$}\rput(-4.5,-0.){\scriptsize $e_z^2$}
 \rput(-4.,-1.25){\scriptsize $e_u^1$}\rput(-4.,-.3){\scriptsize $e_u^0$}\rput(-3.4,-0.3){\scriptsize $e_u^2$}

 \rput(-5,-1.8){\scriptsize $K_4$} \rput(-2.8,0.29){$\circledR$}
%--------$K_3$------------
\cnode(-1,-0.29){3pt}{1}\cnode(-1.5,0.58){3pt}{2}\cnode(-2,-0.29){3pt}{3}
\ncline{1}{2}\ncline{2}{3}\ncline{1}{3}
 \rput(-1.5,.9){\scriptsize$0$} \rput(-2.25,-0.3){\scriptsize$2$} \rput(-0.7,-0.3){\scriptsize$1$}
 \rput(-1.5,-.8){\scriptsize $C_3$} \rput(0,0.29){$=$}

%--------$K_4\textregistered K_3$------------

\cnode(4,-0.29){3pt}{a1}\cnode(3.5,0.58){3pt}{a2}\cnode(3,-0.29){3pt}{a3}
\ncline{a1}{a2}\ncline{a2}{a3}\ncline{a1}{a3}
\cnode(5.232,-1){3pt}{b1}\cnode(5.732,-1.866){3pt}{b2}\cnode(6.232,-1){3pt}{b3}
\ncline{b1}{b2}\ncline{b2}{b3}\ncline{b1}{b3}
\cnode(3.5,2){3pt}{c1}\cnode(4,2.866){3pt}{c2}\cnode(3,2.866){3pt}{c3}
\ncline{c1}{c2}\ncline{c2}{c3}\ncline{c1}{c3}
\cnode(1.768,-1){3pt}{d1}\cnode(0.768,-1){3pt}{d2}\cnode(1.268,-1.866){3pt}{d3}
\ncline{d1}{d2}\ncline{d2}{d3}\ncline{d1}{d3}

\ncline[linewidth=1.6pt]{a1}{b1}\ncline[linewidth=1.6pt]{a2}{c1}\ncline[linewidth=1.6pt]{a3}{d1}
\ncline[linewidth=1.6pt]{b3}{c2}\ncline[linewidth=1.6pt]{c3}{d2}\ncline[linewidth=1.6pt]{d3}{b2}
 \rput(2.9,3.2){\scriptsize $(x,0)$}\rput(4.1,3.2){\scriptsize $(x,1)$}\rput(3.95,1.8){\scriptsize $(x,2)$}
 \rput(.3,-1.){\scriptsize $(y,0)$}\rput(2.25,-1.1){\scriptsize $(y,1)$}\rput(.8,-2.0){\scriptsize $(y,2)$}
 \rput(4,.6){\scriptsize $(z,0)$}\rput(4.5,-0.15){\scriptsize $(z,2)$}\rput(2.5,-0.2){\scriptsize $(z,1)$}
 \rput(4.8,-1.2){\scriptsize $(u,0)$}\rput(6.75,-1.){\scriptsize $(u,2)$}\rput(6.2,-2){\scriptsize $(u,1)$}

\rput(3.5,-2.6){\scriptsize $K_4\circledR C_3$}

\end{pspicture}
\caption{\label{f3}{\footnotesize  $K_4\textregistered C_3$.}}
\end{center}
\end{figure}

By Definition~\ref{Def2.4}, we can obtain the following proposition.

\begin{proposition}\label{prop2.5}\quad
$G_1\textregistered G_2$ is $(\delta_2+1)$-regular and has
$n\,\delta_1$ vertices. Moreover, the vertex-set of
$G_1\textregistered G_2$ can be partitioned into
 $$%\begin{equation}\label{e4.3}
 \text{$\{X_1,X_2,\dots,X_n\}$ such that
 $G[X_i]\cong G_2$ for each $i\in I_n$.}
 $$%\end{equation}
 \end{proposition}

The {\it inflation or inflated graph} of $G$ is a graph obtained
from $G$ by replacing each vertex $x$ by a complete graph
$K_{d_G(x)}$ and joining each edge to a different vertex of
$K_{d_G(x)}$. Inflation graphs have been studied by several authors
(for example, see~\cite{dh96,f98,f01,ksc04,lz14,p00}). Clearly, if
$G$ is $n$-regular then $G\textregistered K_n$ is the inflation
graph of $G$.  In special, Liu and Zhang~\cite{lz14} showed that
$Q_n\textregistered K_n$ is a Cayley graph.

The {\it lexicographic product} $G_1[G_2]$ of two graphs $G_1$ and
$G_2$ is a graph with vertex-set $V(G_1)\times V(G_2)$, and in which
two vertices $(x,i)$ and $(y,j)$ are adjacent if and only if either
$x=y$ and $ij\in E(G_2)$ or $xy\in E(G_1)$, without the condition
``$e_x^i=xy=e_y^j$". Thus, the replacement product graph
$G_1\textregistered G_2$ is a subgraph of the lexicographic product
graph $G_1[G_2]$. In special, Li {\it et al}~\cite{lwxz11} showed
that $G_1[G_2]$ is a Cayley graph if $G_1$ and $G_2$ are Cayley
graphs.

The replacement product of two graphs is an important constructing
method, which can obtain a larger graph from two smaller graphs, and
so it has been widely used to address many fundamental problems in
such areas as graph theory, combinatorics, probability, group
theory, in the study of expander graphs and graph-based coding
schemes~\cite{alw01,don13, gro83,hlw06,krs03,ksr08,rvw02}. The
replacement product has been also used in the designing of an
interconnection networks. For example, the well-known
$n$-dimensional cube-connected cycle $CCC_n$ is a replacement
product $Q_n\textregistered C_n$, where $Q_n$ is a hypercube and
$C_n$ is a cycle of length $n$ (see Preparata and
Vuillemin~\cite{p81}). The graph shown in Figure~\ref{f4} is
$Q_3\textregistered C_3=CCC_3$. In addition, $n$-dimensional
hierarchical hypercube is a replacement product
$Q_{2^n}\textregistered Q_n$ (see Malluhi and Bayoumi~\cite{mb94}).

\begin{figure}[h]
\begin{center}
\psset{unit=.9}
\begin{pspicture}(-5,-4)(7,4)

 \cnode(-5,3){3pt}{000}\cnode(-2,3){3pt}{001}
 \cnode(-5,0){3pt}{010}\cnode(-2,0){3pt}{011}

 \cnode(-4,2.15){3pt}{100}\cnode(-3,2.15){3pt}{101}
 \cnode(-4,0.85){3pt}{110}\cnode(-3,.85){3pt}{111}
\ncline{000}{010}\ncline{000}{001}\ncline{000}{100}
\ncline{011}{010}\ncline{011}{001}\ncline{011}{111}
\ncline{101}{111}\ncline{101}{001}\ncline{101}{100}
\ncline{110}{010}\ncline{110}{100}\ncline{110}{111}
 \rput(-5,3.3){\scriptsize$000$}\rput(-2,3.3){\scriptsize$001$}
 \rput(-5,-.3){\scriptsize$010$}\rput(-2,-.3){\scriptsize$011$}
 \rput(-4.4,2.0){\scriptsize $100$}\rput(-2.55,2.0){\scriptsize $101$}
 \rput(-4.45,1.){\scriptsize$110$}\rput(-2.55,1.){\scriptsize$111$}
 \rput(-3.5,-.5){\scriptsize $Q_3$}
 \rput(3.5,-4.){\scriptsize $Q_3\circledR C_3$}

 \cnode(-3.5,-1.5){3pt}{0}\cnode(-4.5,-3.3){3pt}{1}\cnode(-2.5,-3.3){3pt}{2}
 \ncline{0}{1}\ncline{0}{2}\ncline{2}{1}
 \rput(-3.8,-1.5){\scriptsize $0$}\rput(-4.75,-3.3){\scriptsize$1$}
 \rput(-2.25,-3.3){\scriptsize $2$} \rput(-3.5,-3.8){\scriptsize $C_3$}

\cnode(1.2,3.3){3pt}{0001}\cnode(5.8,3.3){3pt}{0011}
\cnode(0.2,2.3){3pt}{0002}\cnode(1.2,2.3){3pt}{0000}\cnode(5.8,2.3){3pt}{0010}\cnode(6.8,2.3){3pt}{0012}
\cnode(2,1.5){3pt}{1000}\cnode(3,1.5){3pt}{1001}\cnode(4,1.5){3pt}{1011}\cnode(5,1.5){3pt}{1010}
\cnode(2,0.5){3pt}{1002}\cnode(5,0.5){3pt}{1012}
\cnode(2,-0.5){3pt}{1102}\cnode(5,-0.5){3pt}{1112}
\cnode(2,-1.5){3pt}{1100}\cnode(3,-1.5){3pt}{1101}\cnode(4,-1.5){3pt}{1111}\cnode(5,-1.5){3pt}{1110}
\cnode(0.2,-2.3){3pt}{0102}\cnode(1.2,-2.3){3pt}{0100}\cnode(5.8,-2.3){3pt}{0110}\cnode(6.8,-2.3){3pt}{0112}
\cnode(1.2,-3.3){3pt}{0101}\cnode(5.8,-3.3){3pt}{0111}
\ncline{0001}{0011}\ncline{0101}{0111}
\ncline{0002}{0102}\ncline{0012}{0112}
\ncline{1001}{1011}\ncline{1101}{1111}
\ncline{1002}{1102}\ncline{1012}{1112}
\ncline{0000}{1000}\ncline{0010}{1010}
\ncline{0100}{1100}\ncline{0110}{1110}
\ncline{0000}{0001}\ncline{0000}{0002}\ncline{0001}{0002}
\ncline{0100}{0101}\ncline{0100}{0102}\ncline{0101}{0102}
\ncline{0010}{0011}\ncline{0010}{0012}\ncline{0011}{0012}
\ncline{0110}{0111}\ncline{0110}{0112}\ncline{0111}{0112}
\ncline{1000}{1001}\ncline{1000}{1002}\ncline{1001}{1002}
\ncline{1100}{1101}\ncline{1100}{1102}\ncline{1101}{1102}
\ncline{1010}{1011}\ncline{1010}{1012}\ncline{1011}{1012}
\ncline{1110}{1111}\ncline{1110}{1112}\ncline{1111}{1112}
 \rput(1.2,3.65){\scriptsize {(000,2)}}\rput(5.8,3.65){\scriptsize {(001,2)}}
 \rput(-0.5,2.35){\scriptsize {(000,1)}}\rput(1.9,2.35){\scriptsize {(000,0)}}
 \rput(5.1,2.35){\scriptsize {(001,0)}}\rput(7.5,2.35){\scriptsize {(001,1)}}
 \rput(1.3,1.45){\scriptsize {(100,0)}}\rput(2.8,1.85){\scriptsize {(100,2)}}
 \rput(4.2,1.85){\scriptsize {(101,2)}}\rput(5.7,1.45){\scriptsize {(101,0)}}
 \rput(1.3,0.5){\scriptsize {(100,1)}}\rput(5.7,0.5){\scriptsize {(101,1)}}
 \rput(1.3,-0.5){\scriptsize {(110,1)}}\rput(5.7,-0.5){\scriptsize {(111,1)}}
 \rput(1.3,-1.45){\scriptsize {(110,0)}}\rput(2.8,-1.85){\scriptsize {(110,2)}}
 \rput(4.2,-1.85){\scriptsize {(111,2)}}\rput(5.7,-1.45){\scriptsize {(111,0)}}
 \rput(-0.5,-2.35){\scriptsize {(010,1)}}\rput(1.9,-2.35){\scriptsize {(010,0)}}
 \rput(5.1,-2.35){\scriptsize {(011,0)}}\rput(7.5,-2.35){\scriptsize{(011,1)}}
 \rput(1.2,-3.65){\scriptsize {(010,2)}}\rput(5.8,-3.65){\scriptsize {(011,2)}}

\end{pspicture}
\caption{{\footnotesize The cube-connected cycle
$CCC(3)=Q_3\textregistered C_3$.}\label{f4}}
\end{center}
\end{figure}

For simplicity, when a replacement product graph $G_1\textregistered
G_2$ is mentioned, if no otherwise specified, we always assume that
$G_1$ is a $\delta_1$-regular graph with $n$ vertices and $G_2$ is a
$\delta_2$-regular graph with $\delta_1$ vertices. Moreover, we
simply write $\kappa_i=\kappa(G_i)$, $\lambda_i=\lambda(G_i)$ and
$\delta_i=\delta(G_i)$ for each $i=1,2$, and write $xG_2$ for
$\{x\}\times G_2$ for any $x\in V(G_1)$, and let
$I_n=\{1,2,\dots,n\}$.

In this paper, we also need some notations. For a subset $X\subset
V(G)$, use $G[X]$ to denote the subgraph of $G$ induced by $X$. For
two disjoint subsets $X$ and $Y$ in $V(G)$, use $[X,Y]$ to denote
the set of edges between $X$ and $Y$ in $G$. In particular,
$E_G(X)=[X,\overline{X}]$ and let $d_G(X)=|E_G(X)|$, where
$\overline{X} =V(G)\setminus X$.

For a $\lambda'$-connected graph $G$, there is certainly a subset
$X\subset V(G)$ with $|X|\geqslant 2$ such that $E_G(X)$ is a
$\lambda'$-cut and, both $G[X]$ and $G[\overline{X}]$ are connected.
Such an $X$ is called a {\it $\lambda'$-fragment} of $G$. A
$\lambda'$-fragment $X$ of $G$ with minimum cardinality is called a
{\it $\lambda'$-atom} of $G$. The $\lambda'$-atom has been
successfully used in the study of restricted edge-connectivity of
graphs (see, for instance, \cite{m03,oz05,uv03,xx02}).

\section{Edge-connectivity of $G_1\textregistered G_2$}

In this section, we investigate the edge-connectivity of replacement
product graph $G_1\textregistered G_2$. By Definition~\ref{Def2.4},
it is easy to see that if $G_1$ and $G_2$ are connected, then
$G_1\textregistered G_2$ is also connected. We now establish the
upper and lower bounds on the edge-connectivity for replacement
product graphs.

\begin{theorem}\label{thm3.1}\quad
If both $G_1$ and $G_2$ are connected, then
\begin{equation}\label{e3.1}
\min\{\lambda_1,\lambda_2\}\leqslant \lambda(G_1\textregistered
G_2)\leqslant \min\{\lambda_1,\delta_2+1\}.
\end{equation}
Furthermore,
\begin{equation}\label{e3.2}
\min\{\lambda_1,\lambda_2+1\}\leqslant \lambda(G_1\textregistered
G_2)\ \ \text{if $\kappa_1\geqslant 2$}.
\end{equation}
\end{theorem}

\begin{proof}\quad
Let $G=G_1\textregistered G_2$. Clearly,
 \begin{equation}\label{e3.3}
\lambda(G)\leqslant \delta(G)=\delta_2+1.
 \end{equation}

Let $S\subset V(G_1)$ and $E_{G_1}(S)$ be a $\lambda_1$-cut of
$G_1$, and $T=\{(x,i):\ x\in S,\ i\in V(G_2)\}$. Then $E_G(T)$ is an
edge-cut of $G$. Since there is an edge $xy$ in $G_1$ if and only if
there is exactly one edge between $V(xG_2)$ and $V(yG_2)$ in $G$,
$xy\in E_{G_1}(S)$ if and only if there are two vertices $i$ and $j$
of $G_2$ such that $((x,i),(y,j))\in E_G(T)$. Therefore,
$|E_G(T)|=|E_{G_1}(S)|=\lambda_1$ and
 \begin{equation}\label{e3.4}
 \lambda(G)\leqslant |E_G(T)|=\lambda_1.
 \end{equation}

Combining (\ref{e3.3}) with (\ref{e3.4}), we establish the upper
bound on $\lambda(G_1\textregistered G_2)$ in (\ref{e3.1}). We now
show the lower bound in (\ref{e3.1}).

Let $F$ be a $\lambda$-cut in $G$. Then there are two
$\lambda$-fragments associated with $F$ in $G$, say, $X$ and
$\overline{X}$. Let $\{V_1,V_2,\dots,V_n\}$ be a partition of $V(G)$
satisfied property in Proposition~\ref{prop2.5}.

Assume for each $i\in I_n$, either $V_i\subset X$ or $V_i\subset
\overline{X}$. Let $Y=\{i : V_i\subset X, i\in I_n\}$. Then
$Y\subset V(G_1)$, $E_{G_1}(Y)$ is an edge-cut of $G_1$ and
$|E_{G_1}(Y)|=|F|$, and so
\begin{equation}\label{e3.5}
\lambda(G)=|F|=|E_{G_1}(Y)|\geqslant \lambda_1.
 \end{equation}

Assume now that there exists some $i\in I_n$ such that $V_i\cap
X\neq \emptyset$ and $V_i\cap \overline{X}\neq \emptyset$. Then
 \begin{equation}\label{e3.6}
\lambda(G)=|F|\geqslant |[V_i\cap X,V_i\cap \overline{X}]|\geqslant
\lambda(G[V_i])=\lambda(G_2)=\lambda_2.
\end{equation}

Combining (\ref{e3.5}) with (\ref{e3.6}), we establish the lower
bound on $\lambda(G_1\textregistered G_2)$ in (\ref{e3.1}).

%\vskip6pt

To prove (\ref{e3.2}), let $(x,i)$ be any vertex of $V_x\cap X$ and
$(x,j)$ be any vertex of $V_x\cap \overline{X}$. Since $G[V_x]\cong
G_2$ and $G_2$ is $\lambda_2$-connected, there exist $\lambda_2$
edge-disjoint paths $P_1,P_2,\dots,P_{\lambda_2}$ between $(x,i)$
and $(x,j)$ in $G[V_x]\subset G$. Let $(y,k)\in N_G((x,i))$ and
$(z,\ell)\in N_G((x,j))$, where $\{y,z\}\subseteq N_{G_1}(x)$. Since
$\kappa_1\geqslant 2$, there exist at least two internally
vertex-disjoint paths between $y$ and $z$ in $G_1$, one of them
avoids $x$. By the connectedness of $G_2$, there exists a path $Q$
between $(y,k)$ and $(z,\ell)$ in $G$ that avoids the vertices of
$V_x$. Let $P_0=\langle(x,i),Q,(x,j)\rangle$. Thus,
$P_0,P_1,P_2,\dots,P_{\lambda_2}$ are $\lambda_2+1$ edge-disjoint
paths between $(x,i)$ and $(x,j)$. Since $(x,i)\in V_x\cap X$ and
$(x,j)\in V_x\cap \overline{X}$, it is easy to find $|E(P_i)\cap
F|\geqslant 1$ for each $i\in \{0,1,2,\dots,\lambda_2\}$ and so
$$
\lambda(G)=|F|\geqslant \lambda_2+1
$$
as required.
\end{proof}

%\bigskip
Combining the Whitney's inequality $\kappa(G)\leqslant
\lambda(G)\leqslant \delta(G)$ with Theorem~\ref{thm3.1}, we obtain
the following results immediately.

\begin{corollary}\label{cor3.2}\quad
Suppose that both $G_1$ and $G_2$ are connected. Then

{\rm (a)}\ $\lambda(G_1\textregistered G_2)=1$ if $\lambda_1=1$;

{\rm (b)}\ $\lambda(G_1\textregistered G_2)=\lambda_1$ if
$\lambda_2\geqslant \lambda_1$;

{\rm (c)}\ $\lambda(G_1\textregistered G_2)=\lambda_1$ if
$\kappa_1\geqslant 2$ and $\lambda_2\geqslant \lambda_1-1$;

{\rm (d)}\ $\lambda(G_1\textregistered
G_2)=\min\{\lambda_1,\delta_2+1\}$ if $\kappa_1\geqslant 2$ and
$\lambda_2=\delta_2$.
\end{corollary}

\begin{lemma}\label{lem3.3}\quad
$\lambda(G)\leqslant \frac{1}{2}\Delta(G)$ for any connected graph
that contains cut-vertices.
\end{lemma}

\begin{proof}\quad
Suppose that $x$ is a cut-vertex of $G$ and $G-x$ has $k$
components, where $k\geqslant 2$. Then $\lambda(G)\leqslant
\frac{1}{k}|N(x)|\leqslant \frac{1}{2}\Delta(G)$.
\end{proof}

\begin{corollary}\label{cor3.4}\quad
$\lambda(G\textregistered K_n)=\lambda(G)$ for any $n$-regular
connected graph $G$.
\end{corollary}

\begin{proof}\quad
Clearly, $\lambda(G)\leqslant \delta(G)= n$ and
$\lambda(K_n)=\delta(K_n)=n-1$. If $\kappa(G)\geqslant 2$, then
$\lambda(G\textregistered K_n)=\lambda(G)$ by Theorem~\ref{thm3.1}.
If $\kappa(G)=1$, by Lemma~\ref{lem3.3}, then $\lambda(G)\leqslant
\frac{n}{2}<n$ and so $\lambda(K_n)\geqslant \lambda(G)$. By
Corollary~\ref{cor3.2} (b), the result follows.
\end{proof}

\begin{corollary}\label{cor3.5}\quad
$\lambda(G\textregistered C_n)=\min\{\lambda(G),3\}$ for any
2-connected $n$-regular graph $G$.
\end{corollary}

\begin{example}\label{exa3.6}\quad
$\lambda(K_4\textregistered C_3)=\lambda(K_4)=3$, and

$\lambda(CCC_n)=\lambda(Q_n\textregistered
C_n)=\min\{\lambda(Q_n),3\}=\min\{n,3\}=3$ if $n\geqslant 3$.
\end{example}

\begin{remark}\label{rem3.7}\quad
{\rm We conclude this section with a remark on Theorem~\ref{thm3.1}.
The condition ``$\kappa_1\geqslant 2$" in (\ref{e3.2}) is necessary.
For example, two graphs $G_1$ and $G_2$ are shown in
Figure~\ref{f5}. It is easy to see that $\kappa_1=1$, $\lambda_1=4$,
and $\lambda_2=\delta_2=2$, $G_1\textregistered G_2$ is 3-regular,
and
$$
\lambda(G_1\textregistered
G_2)=2<\min\{4,3\}=\min\{\lambda_1,\lambda_2+1\},
$$
which contradicts to the lower bound on $\lambda(G_1\textregistered
G_2)$ given in (\ref{e3.2}).
 }
 \end{remark}

\begin{figure}[h]
\begin{center}
\begin{pspicture}(-.4,-1.5)(2.5,1.4)
\rput{20}{%
\SpecialCoor\degrees[9]
  \multido{\i=0+1}{9}{\cnode(1.5;\i){.11}{1\i}}
 \ncline{10}{18}\ncline{11}{12}\ncline{12}{13}\ncline{13}{14}\ncline{14}{15}\ncline{15}{16}\ncline{16}{17}
 \ncline{10}{13}\ncline{10}{14}\ncline{10}{15}\ncline{10}{16}\ncline{10}{17}
 \ncline{11}{13}\ncline{11}{14}\ncline{11}{15}\ncline{11}{16}\ncline{11}{17}\ncline{11}{18}
 \ncline{12}{14}\ncline{12}{15}\ncline{12}{16}\ncline{12}{17}\ncline{12}{18}\ncline{12}{10}
 \ncline{13}{15}\ncline{13}{16}\ncline{13}{17}\ncline{13}{18}
 \ncline{14}{16}\ncline{14}{17}\ncline{14}{18}
 \ncline{15}{17}\ncline{15}{18}\ncline{16}{18}
 }%
 %\rput(0,-1.2){\scriptsize$S_3$}
 %\rput(-.8,0.8){\scriptsize$123$}\rput(-1.25,-0.){\scriptsize$321$}\rput(-.8,-.8){\scriptsize$231$}
 %\rput(.85,0.8){\scriptsize$213$}\rput(1.3,-0.){\scriptsize$312$}\rput(.8,-.8){\scriptsize$132$}
 \end{pspicture}
\begin{pspicture}(-2.,-1.5)(2,1.4)
\rput{0}{%
\SpecialCoor\degrees[9]
 \multido{\i=0+1}{9}{\cnode(1.5;\i){.11}{2\i}}
 \ncline{20}{21}\ncline{21}{22}\ncline{22}{23}
 \ncline{24}{25}\ncline{26}{27}\ncline{27}{28}\ncline{28}{20}
 \ncline{20}{23}\ncline{20}{24}\ncline{20}{25}\ncline{20}{26}\ncline{20}{27}
 \ncline{21}{23}\ncline{21}{24}\ncline{21}{25}\ncline{21}{26}\ncline{21}{27}\ncline{21}{28}
 \ncline{22}{24}\ncline{22}{25}\ncline{22}{26}\ncline{22}{27}\ncline{22}{28}\ncline{22}{20}
 \ncline{23}{25}\ncline{23}{26}\ncline{23}{27}\ncline{23}{28}
 \ncline{24}{26}\ncline{24}{27}\ncline{24}{28}
 \ncline{25}{27}\ncline{25}{28}\ncline{26}{28}
 }%
 \cnode(-2.3,0){.11}{x}\rput(-2.3,-.3){\scriptsize$x$}
 \ncline{11}{x}\ncline{10}{x}\ncline{17}{x}\ncline{18}{x}
 \ncline{23}{x}\ncline{24}{x}\ncline{25}{x}\ncline{26}{x}
 \rput(-3.,.9){\scriptsize$e^1_x$}\rput(-1.6,.9){\scriptsize$e^8_x$}
 \rput(-3.,.15){\scriptsize$e^2_x$}\rput(-1.6,.15){\scriptsize$e^7_x$}
 \rput(-2.8,-.05){\scriptsize$e^3_x$}\rput(-1.7,-.15){\scriptsize$e^6_x$}
 \rput(-3.,-.9){\scriptsize$e^4_x$}\rput(-1.6,-.9){\scriptsize$e^5_x$}
 \rput(2.6,0){\scriptsize$\circledR$}
 \rput(-2.3,-1.5){\scriptsize$G_1$}\rput(4.8,-1.5){\scriptsize$G_2$}
 \end{pspicture}
\begin{pspicture}(-2.5,-1.5)(1,1.4)
\rput{23}{%
\SpecialCoor\degrees[8]
 \multido{\i=0+1}{8}{\cnode(1.;\i){.11}{3\i}}
 \ncline{30}{31}\ncline{31}{32}\ncline{32}{33}\ncline{33}{34}\ncline{34}{35}\ncline{35}{36}\ncline{36}{37}\ncline{37}{30}
 }%
 \rput(-0.5,1.2){\scriptsize$1$}\rput(0.5,1.2){\scriptsize$8$}
 \rput(-1.15,0.5){\scriptsize$2$}\rput(1.15,0.5){\scriptsize$7$}
 \rput(-1.15,-0.5){\scriptsize$3$}\rput(1.15,-0.5){\scriptsize$6$}
 \rput(-0.5,-1.2){\scriptsize$4$}\rput(0.5,-1.2){\scriptsize$5$}
\end{pspicture}

\caption{\label{f5}\footnotesize {Two graphs $G_1$ and $G_2$ in
Remark~\ref{rem3.7}.}}
\end{center}
\end{figure}
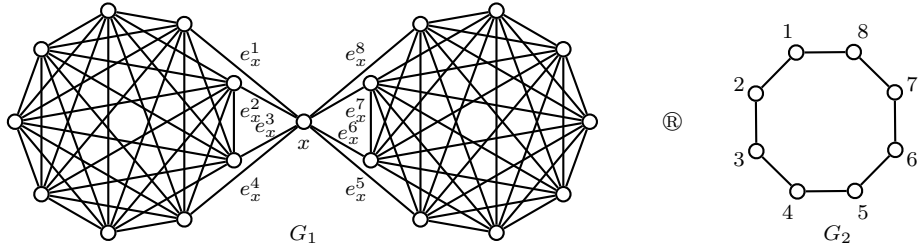

\section{Restricted edge-connectivity of $G_1\textregistered G_2$}

In this section, we investigate the restricted edge-connectivity of
the replacement product of two regular graphs.

\begin{theorem}\label{thm4.1}\quad
If both $G_1$ and $G_2$ are connected, then
$$
\lambda'(G_1\textregistered
G_2)\leqslant\min\{\lambda_1,2\delta_2\}.
$$
\end{theorem}

\begin{proof}\quad
Let $G=G_1\textregistered G_2$. Since $G$ is $(\delta_2+1)$-regular
and $\delta_2+1\geqslant 2$, it is easy to see that $G$ is
$\lambda'$-connected. By Theorem~\ref{thm1.1}.
 \begin{equation}\label{e4.1}
\lambda'(G) \leqslant \xi(G)=2\delta_2.
 \end{equation}
Let $X\subset V(G_1)$ such that $[X,\overline{X}]_{G_1}$ be a
$\lambda_1$-cut of $G_1$. Then $G_1[X]$ and $G_1[\overline{X}]$ are
both connected. Let $Y=\{(x,i):\ x\in X,\ i\in V(G_2)\}$. Then
$G[Y]$ and $G[\overline{Y}]$ are connected, and $|Y|\geqslant
|V(G_2)|=\delta_1\geqslant 2$, $|\overline{Y}|\geqslant |V(G_2)|=
\delta_1\geqslant 2$. Hence, $[Y,\overline{Y}]_G$ is a restricted
edge-cut of $G$. There is an edge $xy$ in $G_1$ if and only if there
is exactly one edge between $V(xG_2)$ and $V(yG_2)$ in $G$, so
$xy\in [X,\overline{X}]_{G_1}$ if and only if there are two vertices
$i$ and $j$ of $G_2$ such that $((x,i),(y,j))\in
[Y,\overline{Y}]_G$. Therefore,
$|[Y,\overline{Y}]_G|=|[X,\overline{X}]_{G_1}|=\lambda_1$ and
 \begin{equation}\label{e4.2}
\lambda'(G)\leqslant |[Y,\overline{Y}]_G|=\lambda_1.
 \end{equation}
Combining (\ref{e4.1}) with (\ref{e4.2}), the result follows.
\end{proof}

\begin{theorem}\label{thm4.2}\quad
$\lambda'(G\textregistered K_n)=\lambda(G)$ for any $n$-regular
connected graph $G$.
\end{theorem}

\begin{proof}\quad
By Corollary~\ref{cor3.4}, $\lambda(G\textregistered
K_n)=\lambda(G)$. By Theorem~\ref{thm4.1},
$\lambda'(G\textregistered K_n)\leqslant\lambda(G)$, and so
$$
\lambda(G)=\lambda(G\textregistered K_n)\leqslant
\lambda'(G\textregistered K_n)\leqslant \lambda(G).
$$
The result follows.
\end{proof}

%\bigskip
For $\delta_1\leqslant 3$, Theorem~\ref{thm4.2} shows that
$\lambda'(G_1\textregistered G_2)=\lambda(G_1)$. In the following
discussion, we always assume $\delta_1\geqslant 4$.

\begin{lemma}\label{lem4.3}\quad
Suppose that both $G_1$ and $G_2$ are connected and
$\delta_1\geqslant 4$, $F$ be a $\lambda'$-cut of
$G_1\textregistered G_2$ and $\{X_1,X_2,\dots,X_n\}$ be a partition
of $V(G_1\textregistered G_2)$ satisfied property in
Proposition~\ref{prop2.5}. If there is some $i\in I_n$ such that
$G[X_i]$ is disconnected in $G-F$, then
 \begin{equation}\label{e4.3}
 \lambda'(G_1\textregistered G_2)\geqslant\min\{\kappa_1+\lambda_2-1,2\lambda_2,\lambda_2'+2\}.
 \end{equation}
\end{lemma}

\begin{proof}\quad
Let $G=G_1\textregistered G_2$. Since $F$ is a $\lambda'$-cut of
$G$, there is some $X\subset V(G)$ with $|X|\geqslant 2$ such that
$F=E_G(X)$. Without loss of generality assume
$|\overline{X}|\geqslant |X|$.

If there exist two distinct $j, k\in I_n$ such that $G[X_j]$ and
$G[X_k]$ are disconnected in $G-F$, then
\begin{equation}\label{e4.4}
|F|\geqslant \lambda(G[X_j])+\lambda(G[X_k])=2\lambda_2.
\end{equation}

Now assume that there exists exactly one integer, say $j\in I_n$,
such that $G[X_j]$ is disconnected in $G-F$. Then $X_j\cap X\neq
\emptyset$ and $X_j\cap \overline{X}\neq \emptyset$. Consider the
following two cases.

{\it Case 1.} $X\subset X_j$.

In this case, $\overline{X}=(V(G)\setminus X_j)\cup (X_j\setminus
X)$. Thus
\begin{align*}
|F|= & |[X, \overline{X}]|\\
   = & |[X, (V(G)\setminus X_j)]|+|[X, X_j\setminus X]|\\
   = & |X|+|[X, X_j\setminus X]|
\end{align*}
If $|X|=\delta_1-1$, then $|[X, X_j\setminus X]|=\delta_2$, and so
$$
|F|=\delta_1-1+\delta_2\geqslant 2\delta_2\geqslant 2\lambda_2.
$$
If $2 \leqslant |X|\leqslant \delta_1-2$, then $[X, X_j\setminus X]$
is a restricted edge-cut of $G[X_j]$, and so
$$
|F|\geqslant |X|+\lambda'(G[X_j])\geqslant \lambda_2'+2.
$$
Hence, in this case,
\begin{equation}\label{e4.5}
|F|\geqslant \min\{2\lambda_2, \lambda_2'+2\}.
\end{equation}

{\it Case 2.} $X\nsubseteq X_j$.

Since $|\overline{X}|\geqslant |X|$, $\overline{X}\nsubseteq X_j$.
Equivalently, there exist at least two sets $X_k$ and $X_{\ell}$
other than $X_j$ such that $X_k\subset X$ and $X_{\ell}\subset
\overline{X}$. Since $\kappa(G_1-u)\geqslant \kappa_1-1 \geqslant 0$
for any vertex $u\in V(G_1)$, there are at least $\kappa_1-1$
internally vertex-disjoint paths between any two distinct vertices
$x$ and $y$ in $G_1-u$. By the definition of $G$, it is easy to see
that there are at least $\kappa_1-1$ internally vertex-disjoint
paths $P_1,P_2,\dots,P_{\kappa_1-1}$ between $X_k$ and $X_{\ell}$ in
$G-X_j$. Let $F'=F\setminus [X_j\cap X, X_j\cap \overline{X}]$.
Since $X_k\subset X$ and $X_{\ell}\subset \overline{X}$,
$|E(P_i)\cap F'|\geqslant 1$ for each $i\in
\{1,2,\dots,\kappa_1-1\}$ and $|F'|\geqslant \kappa_1-1$. Thus, we
have
\begin{align*}
|F|=|[X, \overline{X}]| & = |F'|+|[X_j\cap X, X_j\cap \overline{X}]|\\
                        & \geqslant \kappa_1-1 + \lambda(G[X_j])\\
                        & =\kappa_1+\lambda_2-1,
\end{align*}
that is,
\begin{equation}\label{e4.6}
|F|\geqslant \kappa_1+\lambda_2-1.
\end{equation}
Note that if $\kappa_1=1$, then $|F|\geqslant |[X_j\cap X, X_j\cap
\overline{X}]|\geqslant \lambda_2$ and so (\ref{e4.6}) also holds.

By (\ref{e4.4}), (\ref{e4.5}) and (\ref{e4.6}), the inequality
(\ref{e4.3}) is established.
\end{proof}

\begin{theorem}\label{thm4.4}\quad
Suppose that both $G_1$ and $G_2$ are connected and
$\delta_1\geqslant 4$. Then
 \begin{equation}\label{e4.7}
\min\{\lambda_1,\kappa_1+\lambda_2-1,2\lambda_2,\lambda_2'+2\}\leqslant
\lambda'(G_1\textregistered
G_2)\leqslant\min\{\lambda_1,2\delta_2\}.
 \end{equation}
Furthermore, if $\kappa_1\geqslant \lambda_1-\lambda_2+1$ (or
$\kappa_1\geqslant \lambda_2+1$) and $G_2$ is $\lambda'$-optimal,
then
  \begin{equation}\label{e4.8}
\lambda'(G_1\textregistered G_2)=\min\{\lambda_1,2\delta_2\}.
 \end{equation}
\end{theorem}

\begin{proof}\quad
Let $G=G_1\textregistered G_2$. By Theorem~\ref{thm4.1}, we only
need to show the lower bound on $\lambda'(G_1\textregistered G_2)$
in (\ref{e4.7}). To the end, let $\{X_1,X_2,\dots,X_n\}$ be a
partition of $V(G)$ satisfied property in Proposition~\ref{prop2.5}
and $F$ be a $\lambda'$-cut of $G$. There is some $X\subset V(G)$
with $|X|\geqslant 2$ such that $F=E_G(X)$. Without loss of
generality assume $|\overline{X}|\geqslant |X|$.

By Lemma~\ref{lem4.3}, we only need to show that
$\lambda'(G)\geqslant\lambda_1$ if $G[X_i]$ is connected in $G-F$
for each $i\in I_n$.

In this case, either $X_i\subset X$ or $X_i\subset \overline{X}$ for
each $i\in I_n$. Thus, we can assume $F=E_G(Y\times V(G_2))$, where
$Y\subset V(G_1)$. By the definition of $G$, $|F| = |E_G(Y\times
V(G_2))| = |E_{G_1}(Y)|$. Since $|E_{G_1}(Y)|\geqslant \lambda_1$,
we have $|F|\geqslant \lambda_1$, and so the lower bound on
$\lambda'(G_1\textregistered G_2)$ in (\ref{e4.7}) is established.

%\vskip6pt

We now show the equality (\ref{e4.8}). If $\kappa_1\geqslant
\lambda_1-\lambda_2+1$ (or $\kappa_1\geqslant \lambda_2+1$) and
$G_2$ is $\lambda'$-optimal, then $\lambda_2'=\xi(G_2)=2\delta_2-2$,
and so $\lambda_2=\delta_2$. Thus, we have
$$
\kappa_1+\lambda_2-1 \geqslant \lambda_1 ~(\text{or~}
\kappa_1+\lambda_2-1 \geqslant 2\lambda_2),
$$
and so
\begin{equation}\label{e4.9}
\min\{\lambda_1,\kappa_1+\lambda_2-1,2\lambda_2,\lambda_2'+2\}=\min\{\lambda_1,2\delta_2\}.
\end{equation}
Comparing (\ref{e4.7}) with (\ref{e4.9}), the equality (\ref{e4.8})
is established.
\end{proof}

%\vskip6pt

Note that if $G_2$ is $\lambda'$-optimal then
$\delta_1=|V(G_2)|\geqslant 4$, and so $\lambda'(G_1\textregistered
G_2)$ is well-defined. By Theorem~\ref{thm4.4}, we obtain the
following corollary immediately.

\begin{corollary}\label{cor4.5}\quad
Assume $G_1$ and $G_2$ are connected. If $\kappa_1 = \lambda_1$ and
$G_2$ is $\lambda'$-optimal, then
$$
\lambda'(G_1\textregistered G_2)=\min\{\lambda_1,2\delta_2\}.
$$
\end{corollary}

A connected graph $G$ is {\it super-$\lambda$} if every
$\lambda$-cut isolates a vertex in $G$. It is clear that $G$ is
super-$\lambda$ if and only if $\lambda'(G)>\lambda(G)$. By
Theorem~\ref{thm4.4}, we obtain the following results immediately.

\begin{theorem}\label{thm4.6}\quad
Suppose that $G_1$ and $G_2$ are two connected graphs. If $\kappa_1
\geqslant \lambda_1-\lambda_2+1\geqslant 2$ (or $\kappa_1\geqslant
\lambda_2+1$) and $G_2$ is $\lambda'$-optimal, then

{\rm (a)}\ $G_1\textregistered G_2$ is $\lambda'$-optimal if and
only if $\lambda_1\geqslant 2\delta_2$;

{\rm (b)}\ $G_1\textregistered G_2$ is super-$\lambda$ if and only
if $\lambda_1 > \delta_2+1$.
\end{theorem}

\begin{proof}\quad
Let $G=G_1\textregistered G_2$. Clearly, $\xi(G)=2\delta_2$, and
$\lambda_2=\delta_2\geqslant 2$ since $G_2$ is $\lambda'$-optimal.

Since $\kappa_1\geqslant \lambda_1-\lambda_2+1$ (or
$\kappa_1\geqslant \lambda_2+1$) and $G_2$ is $\lambda'$-optimal, by
Theorem~\ref{thm4.4} we have that
  \begin{equation}\label{e4.10}
\lambda'(G)=\min\{\lambda_1,2\delta_2\}.\qquad\qquad
 \end{equation}
Thus, $G$ is $\lambda'$-optimal if and only if
$\lambda'(G)=\xi(G)=2\delta_2$, that is $\lambda_1\geqslant
2\delta_2$ from (\ref{e4.10}).

Since $\kappa_1\geqslant 2$ and $\lambda_2=\delta_2$, by
Corollary~\ref{cor3.2} (d) we have that
   \begin{equation}\label{e4.11}
 \lambda(G)=\min\{\lambda_1,\delta_2+1\}.\qquad\qquad
 \end{equation}
Note  $2\delta_2 > \delta_2+1$ for $\delta_2\geqslant 2$. It follows
that $G$ is super-$\lambda$ if and only if $\lambda'(G)>\lambda(G)$,
that is $\lambda_1 > \delta_2+1$ from (\ref{e4.11}).
\end{proof}

\begin{corollary}\label{cor4.6}\quad
Assume $G_1$ and $G_2$ are two connected graphs with
$\delta_1\geqslant 4$. If $\kappa_1 = \lambda_1\geqslant 2$ and
$G_2$ is $\lambda'$-optimal, then

{\rm (a)}\ $G_1\textregistered G_2$ is $\lambda'$-optimal if and
only if $\lambda_1\geqslant 2\delta_2$;

{\rm (b)}\ $G_1\textregistered G_2$ is super-$\lambda$ if and only
if $\lambda_1 > \delta_2+1$.
\end{corollary}

\begin{corollary}\label{cor4.8}\quad
$\lambda'(G\textregistered C_n) = \min\{\lambda(G),4\}$ if $G$ is an
$n$-regular graph and $\kappa(G)\geqslant 3$.
\end{corollary}

\begin{example}\label{exa4.9}\quad
{\rm By Corollary~\ref{cor4.8}, it is easy to see that

$\lambda'(K_4\textregistered
C_3)=\min\{\lambda(K_4),4\}=\min\{3,4\}=3$, and

$\lambda'(CCC_n)=\lambda'(Q_n\textregistered
C_n)=\min\{\lambda(Q_n),4\}=\min\{n,4\}=\left\{\begin{array}{ll}
 3\ & {\rm if}\ n=3;\\
 4\ & {\rm if}\  n\geqslant 4.
 \end{array}\right.
 $
 }
\end{example}

\section{Replacement product of Cayley graphs}

In this section, we investigate the restricted edge-connectivity of
the replacement product of two Cayley graphs by a semidirect product
of two groups. We will further confirm that under certain conditions
on the underlying groups and generating sets, the replacement
product of two Cayley graphs is indeed a Cayley graph. Using this
result, we will give a necessary and sufficient condition for such
Cayley graphs to be $\lambda'$-optimal. Based on this condition, we
will construct an example to answer Problem~\ref{prob1.4}.

We first recall the notion of semidirect product of two groups. Let
$A=(A,\circ)$ and $B=(B,\ast)$ be two finite groups. A {\it group
homomorphism} from $A$ to $B$ is a mapping $\phi:\ A \to B$
satisfying $\phi(a\circ b)=\phi(a)\ast\phi(b)$. Let $e_A$ and $e_B$
be identities in $A$ and $B$, respectively, throughout this section.
Group homomorphisms have two important and useful properties.

\begin{proposition}\label{prop5.1}\quad
Let $A$ and $B$ be two finite groups, and $\phi$ be a group
homomorphism from $A$ to $B$. Then

 {\rm (a)}\ $\phi(e_A)=e_B$;

 {\rm (b)}\ $\phi(a^{-1})=(\phi(a))^{-1}$ for any $a\in A$.
\end{proposition}

An {\it action} of $B$ on $A$ is a group homomorphism $\phi:\ B \to
Aut(A)$ defined by $\phi(b)=\phi_b$ and $\phi(b_1b_2)=
\phi(b_1)\phi(b_2)= \phi_{b_1}\phi_{b_2}$.

The {\it orbit} of $a\in A$ under the action $\phi$ of $B$ is
expressed as $a^B=\{\phi_b(a)\in A:\ b\in B\}$.

\begin{example}\label{exa5.2}\quad
{\rm Let $A=(\mathbb{Z}_2)^n$, $B=\mathbb{Z}_n$, and let $e_i$ be an
element in $A$ defined in (\ref{e2.1}) for each $i=0,1,\ldots,n$.

The action $\phi$ of $B$ on $A$ is defined as follows. For each
$a=a_1a_2\dots a_n\in A$,
 $$
 \phi_i(a)=a_{1-i}a_{2-i}\dots a_{n-i({\rm mod}\,n)}\ \ \text{for each $i=0,\ldots,n-1\in B$}.
 $$

For example, if $a=e_1$, then
 $\phi_i(e_1)=e_{i+1}$ for each $i=0,1,\ldots,n-1$.
Under $\phi$ the orbit $e_1^B=\{e_1, e_2,\cdots,e_n\}$.
 }
 \end{example}

We now introduce the concept of the semidirect product of two finite
groups following Robison~\cite{r03}.

The {\it (external) semidirect product $A\rtimes_\phi B$} of groups
$A$ and $B$ with respect to $\phi$ is the group with set $A\times B
= \{(a, b) : a\in A, b\in B\}$ and binary operation ``$*$"
$$
(a_1, b_1)*(a_2, b_2) = (a_1 \phi_{b_1}(a_2), b_1b_2)\ \ \text{for
any $a_1, a_2\in A$ and $b_1, b_2\in B$}.
$$
The identity is $(e_A, e_B)$. Since $\phi_b\in$ Aut$(A)$ is an
automorphism from $A$ to $A$, by Proposition~\ref{prop5.1} (a)
  \begin{equation}\label{e5.2}
 \phi_b(a)=e_A \Leftrightarrow a=e_A\ \ \text{for any $a\in A$ and $b\in B$},
 \end{equation}
By (\ref{e5.2}), it is easy to verify that the inverse $(a, b)^{-1}$
of $(a,b)$ is $(\phi_{b^{-1}}(a^{-1}), b^{-1})$, that is,
 $$
 (a, b)^{-1}=(\phi_{b^{-1}}(a^{-1}), b^{-1}).
 $$

It is also easy to check that the set $\{(a,e_B):\ a\in A\}$ forms a
normal subgroup of $A\rtimes_\phi B$ isomorphic to $A$, and the set
$\{(e_A, b):\ b\in B\}$ forms a subgroup of $A\rtimes_\phi B$
isomorphic to $B$. Thus, $A\rtimes_\phi B\cong A\rtimes B$, a {\it
semidirect product} of two subgroups $A$ and $B$ of a group
$\Gamma$, where $A$ is normal.

The {\it direct product $A\times B$} is a special case of
$A\rtimes_\phi B$, in which the action $\phi(b)$ is the identity
automorphism of $A$ for any $b\in B$, and so $(a_1, b_1)*(a_2, b_2)
= (a_1a_2, b_1b_2)$. Thus the semidirect product is a generalization
of the direct product of two groups.

Many groups can be expressed as a semidirect product of two groups.
For example, using the semidirect product, Feng~\cite{f06} and
Ganesan~\cite{g13} determined the automorphism groups of some Cayley
graphs generated by transposition sets; Zhou~\cite{z11} determined
the full automorphism group of the alternating group graph. The
semidirect product of groups is also used to prove that some
networks are Cayley graphs. For example, using the semidirect
product, Zhou {\it et al.}~\cite{zcx12} showed that the dual-cube
$DC_n$ is a Cayley graph $C_{(\Gamma\times\Gamma)\rtimes_\phi
\mathbb{Z}_2}(S)$, where $\Gamma =(\mathbb{Z}_2)^n$, the action
$\phi: \mathbb{Z}_2\to {\rm Aut}(\Gamma\times\Gamma)$ is defined by
 $$
 \begin{array}{rl}
 \phi_i(\alpha,\beta)
 =&\left\{\begin{array}{ll}
  (\alpha,\beta) &\ {\rm if}\ i=0;\\
  (\beta,\alpha) &\ {\rm if}\ i=1,
 \end{array}\right.
 \end{array}
 $$
and $S=\{(e_0,e_1,0),\ldots, (e_0,e_n,0),(e_0,e_0,1)\}$.

\begin{asp}\label{asp5.3}\quad
{\it Let $A$ and $B$ be two groups with generating sets $S_A$ and
$S_B$, respectively, $|S_A|=|B|\geqslant 2$, $\phi$ be such an
action of $B$ on $A$ that $S_A=x^B$ for some $x\in S_A$, and
$S=\{(e_A,b): \ b \in S_B\}\cup\{(x,e_B)\}$. }
\end{asp}

\begin{theorem}\label{thm5.4}\quad  %{\rm (Alon {\it et al.}~\cite{alw01})}
Under Assumption~\ref{asp5.3}, $S$ generates $A\rtimes_\phi B$.
Moreover, if $S_B=S_B^{-1}$ and $x=x^{-1}$, then $S=S^{-1}$ and
$C_{A\rtimes_\phi B}(S)$ is a replacement product of $C_A(S_A)$ and
$C_B(S_B)$.
\end{theorem}

\begin{remark}\label{rem5.5}\quad
{\rm Before proving this result, we make some remarks on the
theorem.

 (a) \
Since Cayley graphs under our discussion are undirected, by the
definition of Cayley graphs, it is clear that the conditions
``$S_A=S_A^{-1}$, $S_B=S_B^{-1}$ and $S=S^{-1}$" are necessary to
guarantee that Cayley graphs $C_A(S_A)$, $C_B(S_B)$ and
$C_{A\rtimes_\phi B}(S)$ are undirected. By
Proposition~\ref{prop5.1} (b) for any action $\phi$ of $B$ on $A$,
 $$
 (x,e_B)^{-1}=(\phi_{e_B}(x^{-1}),e_B)=(x^{-1},e_B).
 $$
Thus, the condition ``$S=S^{-1}$" means that
 $$
 \begin{array}{rl}
 \{(e_A,b): \ b \in S_B\}\cup\{(x,e_B)\}
 &=(\{(e_A,b): \ b \in S_B\}\cup\{(x,e_B)\})^{-1}\\
 &= \{(e_A,b^{-1}): \ b \in S_B\}\cup\{(x^{-1},e_B)\},
 \end{array}
 $$
which implies that the condition ``$S=S^{-1}$" is equivalent to the
condition ``$S_B=S_B^{-1}\ {\rm and}\ x=x^{-1}$".

Furthermore, since $S_A=x^B$ under the action $\phi$, for any $a\in
S_A$, there is some $b\in B$ such that $a=\phi_b(x)$. By
Proposition~\ref{prop5.1} (b) we have that
 $$
 x=x^{-1} \Leftrightarrow a=\phi_b(x)=\phi_b(x^{-1})=(\phi_b(x))^{-1}=a^{-1}\ \text{for any $a\in
S_A$}.
 $$

(b)\ The original and simple statement of Theorem~\ref{thm5.4} is
due to Alon {\it et al.} (see Theorem 2.3 in~\cite{alw01}, as a
special case of zig-zag products without proof), and a comparatively
complete statement is given by Hoory {\it et al.} (see Theorem 11.22
in~\cite{hlw06}) without the conditions ``$x=x^{-1}$ and
$S_B=S_B^{-1}$", and with an unperfect proof. We give a complete
proof here. }
\end{remark}

%\vskip6pt

\begin{proof}\quad
By the explanation in Remark~\ref{rem5.5} (a), we only need to prove
that $S$ generates $A\rtimes_\phi B$ and $C_{A\rtimes_\phi B}(S)$ is
a replacement product of $C_A(S_A)$ and $C_B(S_B)$.

We first show that $S$ generates $A\rtimes_\phi B$. To the end, we
only need to show that any $(a, b)\in A\rtimes_\phi B$ can be
expressed as products of a sequence of elements of $S$.

By the hypothesis, $S_A$ is a generating set of $A$ and is the orbit
$x^B$ of some $x\in S_A$ under the action $\phi$ of $B$ on $A$.
Since $(a, b)=(a, e_B)*(e_A, b)$, it can be written as a product of
elements from the set $\{ (s_a, e_B):\ s_a\in S_A\}\cup\{(e_A,
s_b):\ s_b\in S_B\}$. Since $S_A=x^B$, for $s_a\in S_A$ there is
some $b\in B$ such that $s_a=\phi_b(x)$, where $b$ can be expressed
as products of a sequence of elements of $S_B$ since $S_B$ is a
generating set of $B$ by the hypothesis. Also since for any $b\in B$
and $\phi_b(x)\in S_A$,
$$
(s_a, e_B)=(\phi_b(x),e_B)=(e_A,b)*(x,e_B)*(e_A,b^{-1}),
$$
the element $(s_a, e_B)$ can be expressed as products of a sequence
of elements of $S$. This implies that $S$ generates the group
$A\rtimes_\phi B$.

We now show that $C_{A\rtimes_\phi B}(S)$ is a replacement product
of $C_A(S_A)$ and $C_B(S_B)$. By Remark~\ref{rem5.5}, under
Assumption~\ref{asp5.3}, Cayley graphs $C_A(S_A)$, $C_B(S_B)$ and
$C_{A\rtimes_\phi B}(S)$ are well-defined and undirected, and so
satisfy the requirements in Definition~\ref{Def2.4}.

Let $(y,i)$ and $(z,j)$ be two distinct vertices in
$C_{A\rtimes_\phi B}(S)$, where $y,z\in A=V(C_A(S_A))$ and $i,j\in
B=V(C_B(S_B))$. Since $C_{A\rtimes_\phi B}(S)$ is a Cayley graph, we
have that
 \begin{equation}\label{e5.3}
\begin{array}{rl}
 (y,i)(z,j)\in E(C_{A\rtimes_\phi B}(S))
&\Leftrightarrow (y,i)^{-1}*(z,j)\\
 &=(\phi_{i^{-1}}(y^{-1}),i^{-1})*(z,j)\\
                 &=(\phi_{i^{-1}}(y^{-1})\phi_{i^{-1}}(z),i^{-1}j)\\
                 &=(\phi_{i^{-1}}(y^{-1}z),i^{-1}j)\\
                 &\in S=\{(e_A,b): \ b \in
                 S_B\}\cup\{(x,e_B)\}.\qquad\qquad
                 \end{array}
 \end{equation}

If $(\phi_{i^{-1}}(y^{-1}z),i^{-1}j)\in\{(e_A,b): b\in S_B\}$, then
$y=z$ by (\ref{e5.2}), and $ij\in E(C_B(S_B))$, which means that the
edge $(y,i)(y,j)$ of $C_{A\rtimes_\phi B}(S)$ is an edge in
$C_A(S_A)\circledR C_B(S_B)$.

If $(\phi_{i^{-1}}(y^{-1}z),i^{-1}j)=(x,e_B)$, then $i=j$ and
$\phi_{i^{-1}}(y^{-1}z)=x$. Since
$\phi_{i^{-1}}\phi_{i}=\phi(i^{-1})\phi(i)=\phi(i^{-1}i)=\phi(e_B)$
is the identity automorphism of $A$, we have
$\phi_{i^{-1}}^{-1}=\phi_{i}$. Thus,
$y^{-1}z=\phi_{i^{-1}}^{-1}(x)=\phi_{i}(x)\in x^B=S_A$, that is,
$z=y\phi_i(x)$ and $yz\in E(C_A(S_A))$. Therefore, if we use $e_y^i$
and $e_z^i$ to label the edge $yz\in C_A(S_A)$ for each
$(y,i)(z,i)\in E(C_{A\rtimes_\phi B}(S))$, that is $yz=e_y^i=e_z^i$,
then the edge $(y,i)(z,i)$ of $C_{A\rtimes_\phi B}(S)$ is an edge in
$C_A(S_A)\circledR C_B(S_B)$.

It follows that the structure of $C_{A\rtimes_\phi B}(S)$ satisfies
the requirements of Definition~\ref{Def2.4}, and so
$C_{A\rtimes_\phi B}(S)$ is a replacement product of $C_A(S_A)$ and
$C_B(S_B)$.
\end{proof}

\begin{example}\label{exa5.6}\quad
{\rm Let $A=(\mathbb{Z}_2)^n$ and $B=\mathbb{Z}_n$. Then $e_A=e_0$
and $e_B=0$. Let $S_A=\{e_1,e_2,\dots,e_n\}$, where $e_i$ is defined
in (\ref{e2.1}), and $e_i^{-1}=e_i$ for each $i\in
\{1,2,\ldots,n\}$, and let $S_B=\pm \{s_1,s_2,\dots,s_k\}$. The
Cayley graph $C_{A}(S_A)$ is a hypercube $Q_n$ by
Example~\ref{exa2.3} and the Cayley graph $C_{B}(S_B)$ is a
circulant graph $G(n,\pm S)$ by Example~\ref{exa2.2}. Let $\phi$ be
the action of $B$ on $A$ defined in Example~\ref{exa5.2}. Then $S_A$
is the orbit $e_1^B$ of $e_1\in S_A$ under $\phi$. Let
$S=\{(e_A,s):\ s\in S_B\}\cup\{(e_1,e_B)\}$. Then $S=S^{-1}$. By
Theorem~\ref{thm5.4}, $S$ generates $A\rtimes_\phi B$, and
$C_{A\rtimes_\phi B}(S)$ is a replacement product of $C_A(S_A)$ and
$C_B(S_B)$.

In special, if $S_B=\{1,n-1\}$, then
$S=\{(e_0,1),(e_0,n-1),(e_1,0)\}$. The Cayley graph
$C_{(\mathbb{Z}_2)^n\rtimes_\phi \mathbb{Z}_n}(S)=Q_n\circledR
C_n=CCC_n$. The cube-connected cycle $CCC(3)$, shown on the right
side in Figure~\ref{f4}, is a replacement product of $Q_3$ and
$C_3$, and is the Cayley graph $C_{\mathbb{Z}^3_2\rtimes_\phi
\mathbb{Z}_3}(\{(000,1),(000,2),(100,0)\})$. }
 \end{example}

A graph $G$ is $\kappa$-optimal if $\kappa(G)=\delta(G)$. The
following theorem presents a necessary and sufficient condition for
a Cayley graph $C_{A\rtimes_\phi B}(S)$ to be $\lambda'$-optimal if
$C_A(S_A)$ is $\kappa$-optimal and $C_B(S_B)$ is $\lambda'$-optimal.

\begin{theorem}\label{thm5.7}\quad
Under Assumption~\ref{asp5.3}, let $S=\{(e_A,s):\ s \in S_B\}
\cup\{(x,e_B)\}$ and $S=S^{-1}$. If Cayley graph $C_A(S_A)$ is
$\kappa$-optimal and Cayley graph $C_B(S_B)$ is $\lambda'$-optimal,
then Cayley graph $C_{A\rtimes_\phi B}(S)$ is $\lambda'$-optimal
$\Leftrightarrow |S_A|\geqslant 2|S_B|$.
\end{theorem}

\begin{proof}\quad
By Theorem~\ref{thm5.4},  $C_{A\rtimes_\phi B}(S)$ is a replacement
product of $C_A(S_A)$ and $C_B(S_B)$. Since $C_A(S_A)$ is
$\kappa$-optimal,
$\kappa(C_A(S_A))=\lambda(C_A(S_A))=\delta(C_A(S_A))=|S_A|\geqslant
2$. Also since $C_B(S_B)$ is $\lambda'$-optimal, by
Corollary~\ref{cor4.6} (a) $C_{A\rtimes_\phi B}(S)$ is
$\lambda'$-optimal if and only if $|S_A|\geqslant 2|S_B|$.
\end{proof}

\begin{example}\label{exm5.8}\quad
{\rm By Example~\ref{exa5.6}, the cube-connected cycle $CCC_n$
$=Q_n\circledR C_n$ is 3-regular, $\xi(CCC_n)=4$,
$|S_{Q_n}|=n\geqslant 2$ and $|S_{C_n}|=2$.
  $$
  |S_A|=\left\{\begin{array}{ll}
 3<4=2|S_B|\ & {\rm if}\ n=3;\\
 n\geqslant 4=2|S_B|\ & {\rm if}\ n\geqslant 4.\qquad\qquad\quad
 \end{array}\right.
 $$
By Example~\ref{exa4.9} and Theorem~\ref{thm5.7}, $CCC_n$ is
 $$
 \left\{\begin{array}{ll}
 \text{not $\lambda'$-optimal ($\lambda'=\lambda=3<4=\xi$) if
 $n=3$;}\\
 \text{$\lambda'$-optimal (i.e.,
 $\lambda'=4=\xi$) if $n\geqslant 4$.}
 \end{array}\right.
 $$
 }
 \end{example}

\begin{theorem}\label{thm5.9}\quad
Let $A=(\mathbb{Z}_2)^n$ and $B=\mathbb{Z}_n$,
$S_A=\{e_1,e_2,\dots,e_n\}$, where $e_i$ is defined in (\ref{e2.1}),
$S_B=\pm \{s_1,s_2,$ $\dots,s_k\}$ with $k\geqslant 2$ and
$s_k<\frac{n}{2}$, $\phi$ be the action of $B$ on $A$ defined in
Example~\ref{exa5.2}. Let $G=C_{A\rtimes_\phi B}(S)$ with order
$\upsilon(G)$, where $S=\{(e_0,s): s \in S_B\}\cup\{(e_1,0)\}$. If
$\frac n2<|S_B|<n-1$, then $G$ is not $\lambda'$-optimal, and
$$
\lambda(G)<\lambda'(G)=n<\frac{\upsilon(G)}{2}\ \ {\rm for}\
n\geqslant 3,
$$
and $G[X]\cong C_B(S_B)$ for any $\lambda'$-atom $X$ of $G$.
\end{theorem}

\begin{proof}\quad
By Example~\ref{exa2.3} $C_A(S_A)\cong Q_n$, by Example~\ref{exa2.2}
$C_B(S_B)\cong G(n; S_B)$, and by Theorem~\ref{thm5.4} the Cayley
graph $G=C_{A\rtimes_\phi B}(S)$ is a replacement product of $Q_n$
and $G(n; S_B)$. Since $k\geqslant 2$ and $s_k<\frac n2$, $G(n;
S_B)$ is $\lambda'$-optimal by Example~\ref{exa2.2}. Since $Q_n$ is
$\kappa$-optimal and $|S_A|=n<2|S_B|$, $G$ is not $\lambda'$-optimal
by Theorem~\ref{thm5.7}. By Corollary~\ref{cor4.5},
$\lambda'(G)=\min\{n,2|S_B|\}=n$. Since $G$ is vertex-transitive and
 $|S_B|<n-1$, we have that
 $$
 \lambda(G)=\delta(G)=|S|=|S_B|+1<n=\lambda'(G).
 $$
Note that $\upsilon(G)=n\cdot 2^n$ and that $k\geqslant 2$ implies
$n\geqslant 5$. It follows that
$$
\lambda(G)<\lambda'(G)=n=\frac{n\,2^n}{2^n}=\frac{\upsilon(G)}{2^n}<\frac{\upsilon(G)}{2}\
{\rm for}\ n\geqslant 3.
$$

%\vskip6pt

We now show the second conclusion. Let $X$ be a $\lambda'$-atom of
$G$ and $F=E_G(X)$. Then $|X|\leqslant \frac{\upsilon(G)}{2}$ and
$F$ is a $\lambda'$-cut of $G$. We need to prove $G[X]\cong
C_B(S_B)$. We first note that
 \begin{equation}\label{e5.4}
 |F|=\lambda'(G)=n<2|S_B|=4k.
 \end{equation}

Let $\{X_1,X_2,\dots,X_n\}$ be a partition of $V(G)$  satisfied
property in Proposition~\ref{prop2.5}. Then $G[X_i]\cong C_B(S_B)$
for each $i\in I_{2^n}$. If there exists some $j\in I_{2^n}$ such
that $G[X_j]$ is disconnected in $G-F$ then, by Lemma~\ref{lem4.3}
and Example~\ref{exa2.2},
\begin{align*}
|F| & \geqslant \min\{\kappa(C_A(S_A))+\lambda(C_B(S_B))-1, 2\lambda(C_B(S_B)), \lambda'(C_B(S_B))+2\}\\
    & =\min\{n+2k-1,4k\}=4k,
\end{align*}
which contradicts with (\ref{e5.4}). It follows that $G[X_i]$ is
connected in $G-F$, that is, either $X_i\subset X$ or $X_i\subset
\overline{X}$ for each $i\in I_{2^n}$.

If both $X$ and $\overline{X}$ contain at least two sets of
$X_1,X_2,\dots,X_{2^n}$, then, by comparing the structure of $G$
with that of $Q_n$, it is easy to see that the subset of edges in
$Q_n$ corresponding to $F$ is a restricted edge-cut of $Q_n$. Hence,
by Example~\ref{exa2.3},
$$
|F|\geqslant\lambda'(Q_n)=2n-2 > n = \lambda'(G)=|F|,
$$
a contradiction. Namely, $X=X_i$ or $\overline{X}=X_i$ for some
$i\in I_{2^n}$.

Since $|X|\leqslant \frac{\upsilon(G)}{2}$, we have $X=X_i$ and
$\overline{X}=V(G)\setminus X_i$ for some $i\in I_{2^n}$. Thus every
$\lambda'$-cut of $G$ isolates a subgraph which is isomorphic to
$C_B(S_B)$. In other words, $G[X]\cong G[X_i]\cong C_B(S_B)$ for
each $i\in I_{2^n}$.
\end{proof}

\begin{remark}\label{rem5.10}\quad
{\rm We make some remarks on the conditions in Theorem~\ref{thm5.9}.

The condition ``$k\geqslant 2$" is necessary. In fact, if $k=1$,
then $C_B(S_B)$ is a cycle $C_n$. By Example~\ref{exm5.8},
 $$
CCC_n \ {\rm is} \left\{\begin{array}{ll}
 \text{not $\lambda'$-optimal and $\lambda'=\lambda=3$ if
 $n=3$};\qquad\\
 \text{$\lambda'$-optimal if $n\geqslant 4$}.
 \end{array}\right.
 $$

The condition ``$|S_B|>\frac{n}{2}$" is necessary.
Theorem~\ref{thm5.7} means that $C_{A\rtimes_\phi B}(S)$ is
 $$
 \text{not $\lambda'$-optimal $\Leftrightarrow |S_A|<2|S_B|$, i.e.,
$|S_B|>\frac 12\,|S_A|=\frac 12\,n$.}
 $$

The condition ``$|S_B|<n-1$" is also necessary. In fact, if
$|S_B|=n-1$ then $G(n;S_B)$ is a complete graph $K_n$ by
Example~\ref{exa2.2}. Thus, $\lambda(G)=n=\lambda'(G)$, which
contradicts to our conclusion. }
\end{remark}

The following theorem gives a straight answer to
Problem~\ref{prob1.4}.

\begin{theorem}\label{thm5.11}\quad
For a given odd integer $d\,(\geqslant 5)$ and any integer $s$ with
$1\leqslant s\leqslant d-3$, there is a Cayley graph $G$ with degree
$d$ such that $\lambda'(G)=d+s<\frac12\,{\upsilon(G)}$.
\end{theorem}

\begin{proof}\quad
In Theorem~\ref{thm5.9}, let $n=d+s$ and $k=\frac{d-1}{2}$, then
$|S_B|=d-1$ and $G=C_{\mathbb{Z}^{d+s}_2\rtimes_\phi
\mathbb{Z}_{d+s}}(S)$ is a Cayley graph. Since $1\leqslant
s\leqslant d-3$, we have $\frac{d+s}{2}<|S_B|<d+s-1$. By
Theorem~\ref{thm5.9}, $G$ is not $\lambda'$-optimal, and
$$
\lambda(G)=d<\lambda'(G)=d+s<\frac{(d+s)\cdot
2^{d+s}}{2}=\frac{\upsilon(G)}{2}.
$$
The theorem follows.
\end{proof}

\section{Conclusion}
In this paper, we investigate the restricted edge-connectivity of
replacement product of two graphs. By means of the semidirect
product two groups, we further confirm that under certain
conditions, the replacement product of two Cayley graphs is also a
Cayley graph, and give a necessary and sufficient condition for such
Cayley graphs to have maximum restricted edge-connectivity. Based on
these results, for given odd integer $d$ and integer $s$ with $d
\geqslant 5$ and $1\leqslant s\leqslant d-3$, we construct a Cayley
graph with degree $d$ whose restricted edge-connectivity is equal to
$d+s$, which answers a problem proposed ten years ago.

In the proof of this result, the replacement product of graphs plays
a key role. Thus, further properties of replacement products deserve
further research.

\vskip10pt

{\bf Acknowledgements}  {This work was supported by National Natural
Science Foundation of China (Grant Nos. 61272008, 11571044),
University Natural Science Research Project of Anhui Province (Grant
No. KJ2016A003) and Scientific Research Fund of Anhui University of
Finance \& Economics (Grant No. ACKY1532).}

\end{document}